\documentclass[11pt,a4paper]{article}

\usepackage[left=2.3cm, top=2.3cm,bottom=2.3cm,right=2.3cm]{geometry}

\usepackage{graphicx}
\usepackage{latexsym}
\usepackage{amsbsy}
\usepackage{amsthm}
\usepackage{amssymb}
\usepackage{amsfonts}
\usepackage{dsfont}
\usepackage[british]{babel}
\usepackage{amsmath}
\usepackage{xcolor}
\usepackage{float}

\def\E{{\mathds E}}
\def\1{{\mathds 1}}
\def\N{{\mathds N}}
\def\P{{\mathds P}}

\def\R{{\mathds{R}}}

\def\cF{{\mathcal F}}

\def\cH{{\mathcal H}}
\def\cI{{\mathcal I}}

\def\g{\gamma}
\def\k{{\kappa}}
\def\l{\lambda}

\def\tr{{\vartriangle}}
\def\bm #1{{\boldsymbol{#1}}}

\def\xi{\bm X}

\makeatletter
\newcommand{\vast}{\bBigg@{4}}
\makeatother

\DeclareMathOperator{\aff}{aff}

\newcommand{\dint}{\mathrm{d}}

\newtheorem{theorem}{Theorem}[section]
\newtheorem{lemma}[theorem]{Lemma}
\newtheorem{conjecture}[theorem]{Conjecture}

{\theoremstyle{definition}
	\newtheorem {remark}[theorem]{Remark}
}

\begin{document}

\title{\bfseries Stars of Empty Simplices}

\author{Matthias Reitzner and Daniel Temesvari}

\date{}

\maketitle

\begin{abstract}
	Let $X\subset \R^d$ be an $n$-element point set in general position. 
	For a $k$-element subset $\{x_{1},\ldots,x_{k}\} \subset X$, $k\leq d$, let the degree $\deg_k(x_{1},\ldots,x_{k})$ be the number of empty simplices $\{x_{1},\ldots,x_{{d+1}}\} \subset X$ containing no other point of $X$. The $k$-degree of the set $X$, denoted $\deg_k(X)$, is defined as the maximum degree  over all $k$-element subset of $X$.
	
	We show that if $X$ is a random point set consisting of $n$ independently and uniformly chosen points from a convex body $K$ then $\deg_d(X)=\Theta(n)$, improving results previously obtained by B\'ar\'any, Marckert and Reitzner \cite{BMR13} and Temesvari \cite{Te18} and giving the correct order of magnitude with a significantly simpler proof. Furthermore, we investigate $\deg_k(X)$. In the case $k=1$ we prove that $\deg_1(X)=\Theta(n^{d-1})$ for $d \geq 3 $.
\end{abstract}

\section{Introduction and main results}

Let $X \subset \R^d$ be a finite point set in general position, i.e., no $k+2$ points of $X$ lie in the same $k$-dimensional affine subspace, for any $k\leq d-1$. By $ X \choose k$ we denote the set of all $k$-element subsets of $X$, by $A^o$ the interior and by $[A]$ the convex hull of a set $A\subset \R^d$. 
We call $\{x_{1},\ldots,x_{d+1}\} \in { X \choose d+1}$ an empty simplex if $[x_{1},\ldots,x_{{d+1}}]^o \cap X = \emptyset$.
For $\{x_{1},\ldots,x_{k}\} \in { X \choose k}$ we define the $k$-degree, 
 denoted by $\deg_k(x_{1},\ldots,x_{k};X)$, as the number of subsets $\{x_{{k+1}},\ldots,x_{{d+1}}\} \in X\setminus \{x_{1},\ldots,x_{k}\}$ such that $\{x_{1},\ldots,x_{{d+1}}\}$ is an empty simplex. This definition can be written as
\begin{align}\label{def:k-deg_subset}
	\deg_k(x_{1},\ldots,x_{k};X) = \sum_{ \{ x_{{k+1}},\ldots,x_{{d+1}} \} \in {X\setminus \{x_{1},\ldots,x_{k}\} \choose d-k+1}} \1([x_{1},\ldots,x_{{d+1}}]^o \cap X = \emptyset).
\end{align}
The union of these $\deg_k(x_{1},\ldots,x_{k};X)$ many empty simplices is what we call a \lq star of empty simplices\rq. 
The $k$-degree of $X$, denoted as $\deg_k(X)$, is defined as the degree of the maximal star, i.e.,
\begin{align}\label{def:k-deg_set}
	\deg_k(X) = \max_{ \{ x_{1},\ldots,x_{k} \} \in {X \choose k} } \deg_k(x_{1},\ldots,x_{k};X).
\end{align}

The quantity $\deg_d(X)$ was introduced by  Erd\H{o}s \cite{Er92}, when posing the question whether in the planar case, i.e., $d=2$, the $d$-degree of a point set $X$ goes to infinity as the number of points goes to infinity. It quickly became a question of interest to B\'ar\'any and K\'arolyi, which formulated this as a conjecture  in \cite{BK01}. This conjecture was later restated in \cite{BMP05} by Brass, Moser and Pach. However, the case of considering a deterministic point set $X \subset \R^2$ proved itself to be already quite intricate to solve and besides B\'ar\'any and K\'arolyi \cite{BK01}, showing that $deg_d(X)\geq 10$ for a sufficiently large number of points, and B\'ar\'any and Valtr \cite{BV04}, giving a construction for a set $X$ with $n$ points in general position such that $\deg_d(X)=4\sqrt{n}(1+o(1))$, no further progress has been made on this, let alone on the general case of $X \subset \R^d$.

In \cite{BMR13} B\'ar\'any, Marckert and Reitzner turned their attention to random point sets $\xi_n \subset \R^2$ consisting of $n$ independent and uniformly chosen points from a convex body $K \subset \R^2$, i.e., a compact, convex set with nonempty interior. They showed that in fact the degree tends to infinity. For sufficiently large $n$ the assertion holds true in expectation, 
$$\E \deg_d (\xi_n) \geq c(d) (\ln n)^{-1} \, n , $$
and there is convergence in probability, i.e., 
$$\deg_d(\xi_n) \overset{\P}{\to} + \infty$$
as $n \to \infty$. 
Observe that this lower bound for the expectation is surprisingly close to the trivial upper bound $(n-d)$, up to a logarithmic factor.
Temesvari \cite{Te18} generalized their proof ideas to arbitrary dimension $d$ and all moments of $\deg_d(\xi_n)$, giving the bound $\E \deg_d(\xi_n)^k \geq c(d) (\ln n)^{-1} n^k$ for sufficiently large $n$.

The first part of this paper is concerned with improving the lower bound for $\deg_d(\xi_n)$. In fact we are able to remove the logarithmic factor completely and determine the asymptotic order  with a significantly simpler proof than in \cite{BMR13} and \cite{Te18}. Thus the expected degree of a uniform random point set is surprisingly large: there is a star of empty simplices where the number of spikes is at least a constant proportion of \emph{all} random points.

\begin{theorem}\label{th:deg_d}
Let  $\xi_n$ be a set of $n$ independent and uniformly chosen random points from a convex body $K\subset \R^d$. Then there exists a constant $c(d,K)>0$ such that 
	\[
	c(d,K) n \leq \E \deg_d (\xi_n) \leq n .
	\]
\end{theorem}

\begin{remark}
In the planer case, computations allow to derive an explicit value for $c(d,B^2)$, namely 
\begin{align*}
0,11  \, n  \  (1+o(1)) \leq \E \deg_2(\xi_n) ,
\end{align*}
see \eqref{eq:constd=2}. 
Numerical computations by M. Meckes and by D. Temesvari suggest that the optimal constant may be surprisingly large. For $K=B^2$ or the planar square the results  suggest that 
$$ \frac {\E \deg_2 (\xi_n) }n \in [0.70, 0.95] . $$
\end{remark}

\begin{remark}\label{deg_d_moments}
	By the trivial bound $\deg_d \xi_n \leq n$, on the one hand, and Jensen's inequality in conjunction with Theorem \ref{th:deg_d}, on the other hand, we immediately get the asymptotic behavior of all the moments of $\deg_d \xi_n$. Namely, we have
	\[
			\E \deg_d (\xi_n)^k = \Theta(n^k)
	\]
	as $n \to \infty$.
\end{remark}

\begin{figure}[H]
	\centering
	\includegraphics[scale=0.41, bb=50 95 420 370, clip=true]{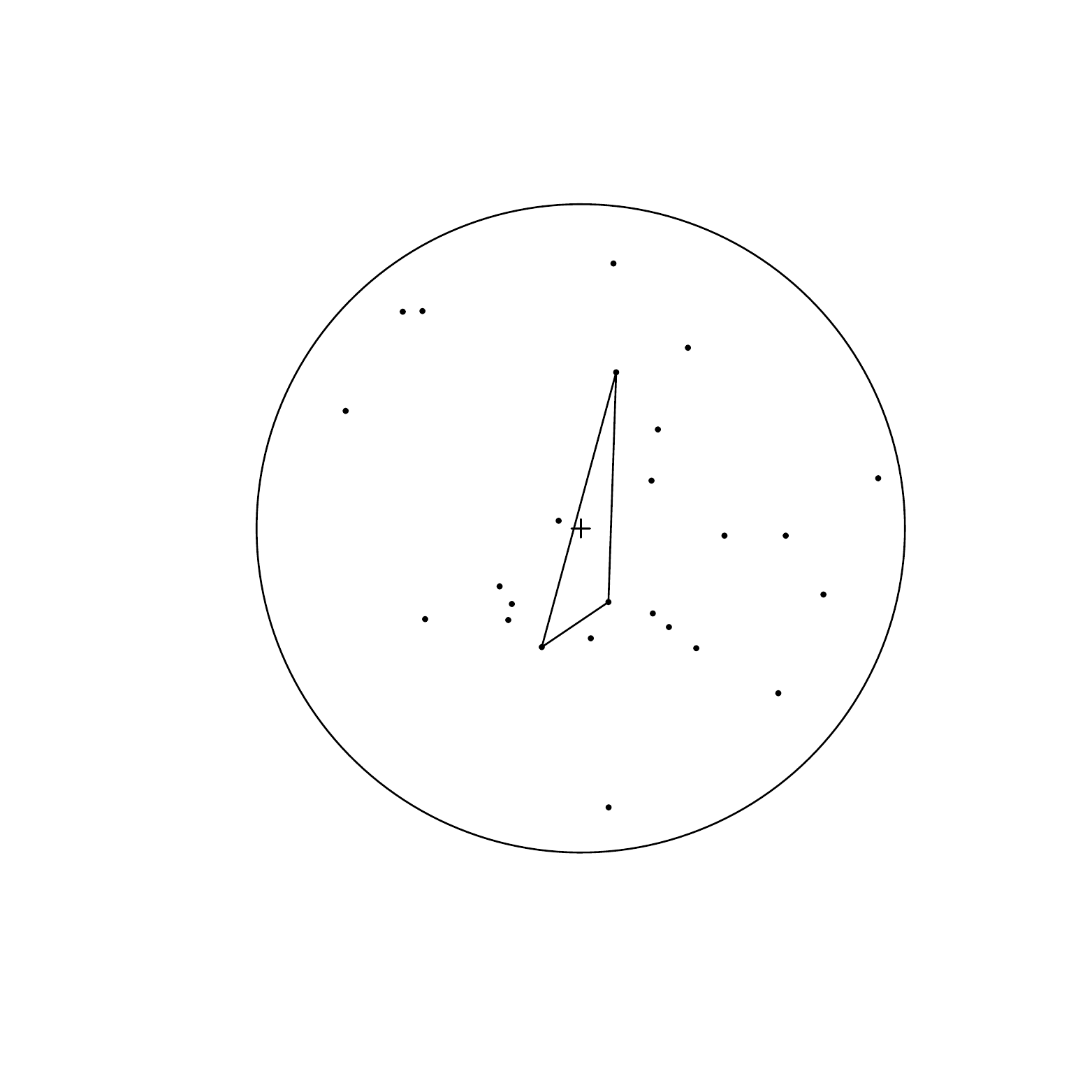}
	\includegraphics[scale=0.41, bb=50 95 420 370, clip=true]{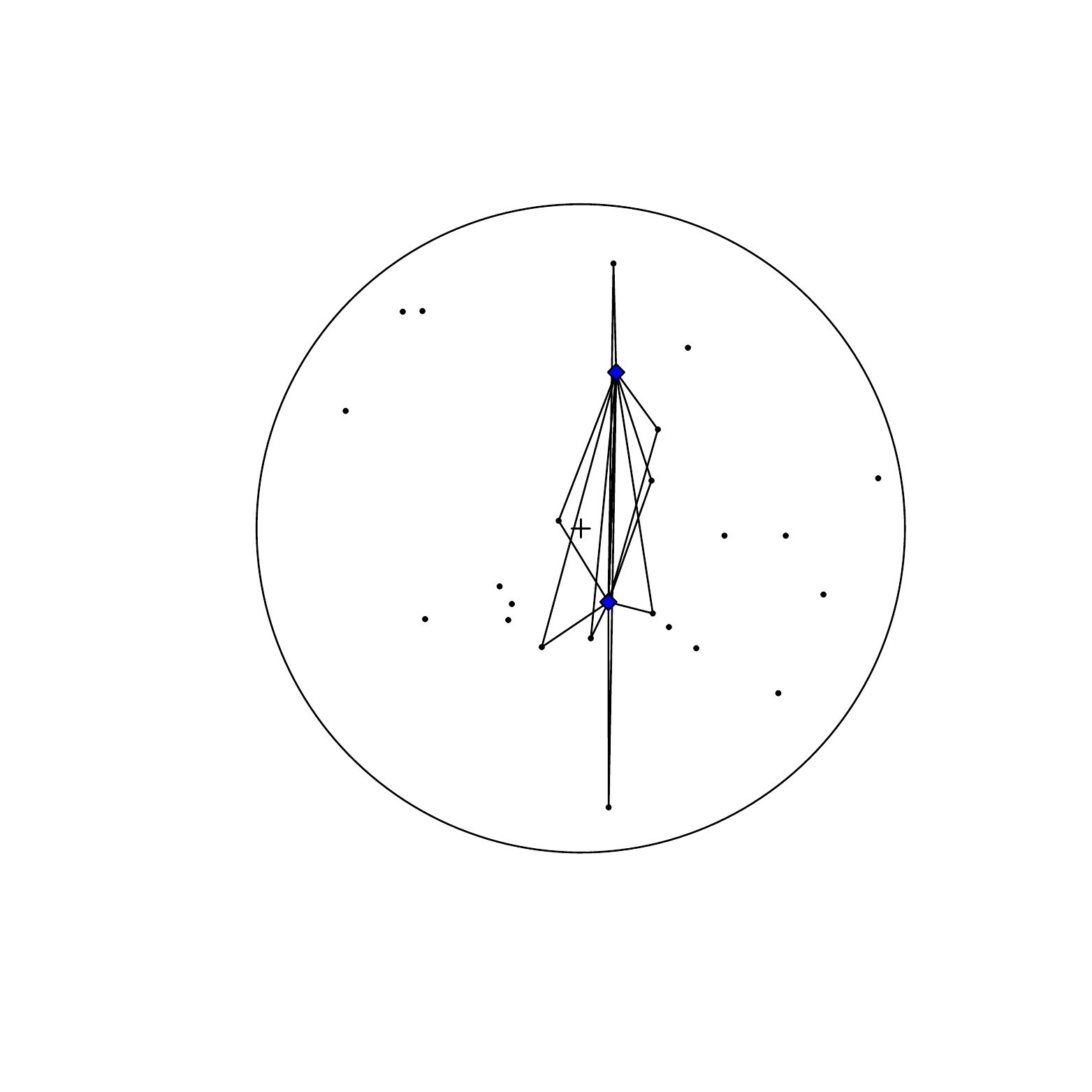}
	\includegraphics[scale=0.41, bb=50 95 420 370, clip=true]{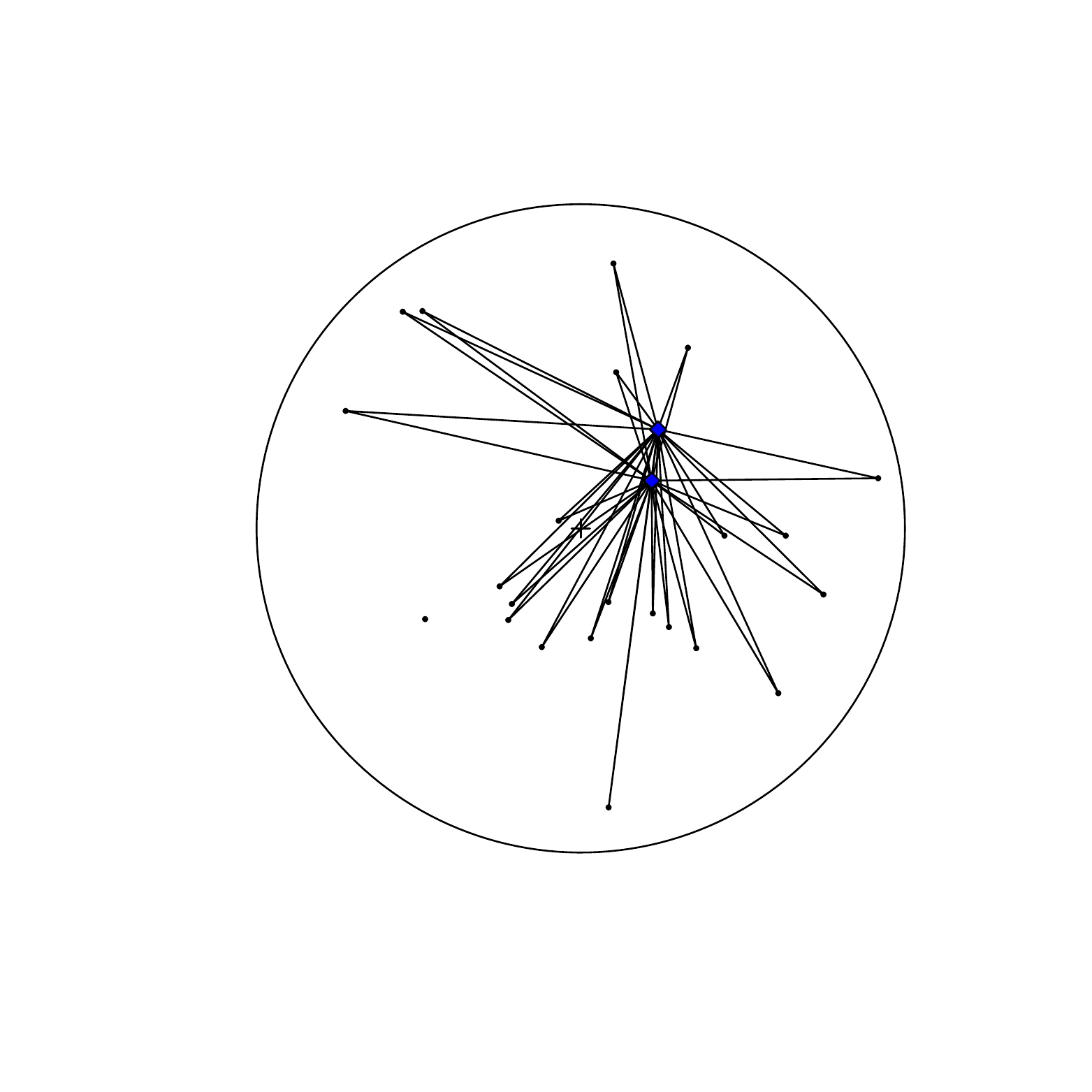}
	\caption{Point set with 25 point. Left: Instance of an empty triangle. Middle: Degree of this basis is 8. Right: Degree of this basis is 22 and is also the degree of the whole point set.}
	\label{simulation2}
\end{figure}

It is of interest to compare the maximal degree to the degree of a typical base, i.e. a typical $d$-tuple of points.
For uniform random points in a planar convex body $K$ the expected number of empty triangles $N_{\tr}(\xi_n)$ is asymptotically bounded by $2n^2$, as was shown by Valtr~\cite{va}. 
In general dimensions $d \geq 3$, a result by B\'ar\'any and F\"uredi \cite{BF87} states that there exists a constant $c(d)>0$ such that the expected number of empty simplices in a uniform random point set satisfies $ N_{\tr} (\xi_n) \leq c(d) n^{d}$. Our next theorem strengthens these results by describing the asymptotic behavior of $N_{\tr}(\xi_n)$, as $n$ goes to infinity, in terms of lower and upper bounds of order $n^d$. Furthermore, in the case $d=2$, this gives rise to an exact asymptotic of $N_{\tr}(\xi_n)$.
In the following, $\kappa_d:=\pi^\frac d2\Gamma\left(\frac d2 +1\right)^{-1}$ is the volume of the $d$-dimensional Eucidean unit ball.

\begin{theorem}\label{th:emptytr}
	Let  $\xi_n$ be a set of $n$ independent and uniformly chosen random points from a convex body $K\subset \R^d$. Then 
$$
\frac 2{d!} \leq 
\lim_{n \to \infty} n^{-d} \E N_\tr (\xi_n)
\leq 
\frac {d  }{(d+1)}  \frac {\k_{d-1}^{d+1}  \k_{d^2}}{\k_d^{d-1}  \k_{(d-1)(d+1)}} ,
$$
for $d \geq 3$, and for $d=2$ 
$$
\lim_{n \to \infty} n^{-2} \E N_\tr (\xi_n)
= 2 .
$$
More precisely, 
\[
	\lim_{n \to \infty} n^{-d} \E N_\tr(\xi_n)= \frac {d \k_d} {(d+1)} \l_d(K)^{-d} \int_{\cH_d} \l_{d-1} (K \cap H) ^{d+1} \dint H  .
\]
\end{theorem}
Here $\l_d$ denotes Lebesgue measure in $\R^d$, and the integration $\int_{\cH_d} \dots \dint H$ on the space $\cH_d$ of affine hyperplanes in $\R^d$ is with respect to the rigid motion invariant Haar measure. 

We want to point out that the resulting inequality 
$$
\int_{\cH_d} \l_{d-1} (K \cap H) ^{d+1} \dint H  
	\geq \frac {2(d+1)}{d!d \k_d} \l_d(K)^{d} 
$$
seems to be new.

In the planar case  we have $ \E N_\tr(\xi_n)= 2 n^2 \ (1+o(1))$, the number of pairs of points is ${n \choose 2}$ and each triangle has three edges. This yields that the degree of a typical pair of points is
$$
\E \deg_2(X_1, X_2; \xi_n) = 12 (1+o(1))
$$
as $n \to \infty$. Since, in general dimensions, there are ${n \choose d}$ simplices of dimension $(d-1)$ and $c(d,K) n^{d} (1+o(1))$ empty $d$-dimensional simplices, the typical degree again is constant,
$$ \E \deg_d(X_1, \dots , X_d; \xi_n) = c(d,K) (1+o(1)) . $$
Here and in the following $c(d)$ and $c(d,K)$ denote generic constants depending on the dimension $d$, respectively the set $K$ and the dimension $d$, whose precise values may differ from line to line. 
Thus, in all dimensions the expected maximal degree $ \E \deg_d (\xi_n) = \Theta(n)$ is surprisingly far from the typical degree, which is a constant.

In the second part of this work we are again concerned with the question regarding the asymptotic behavior of the degree, but this time posed for the newly introduced quantity $\deg_k(\xi_n)$, $k=1, \ldots, d-1$.  Here, one easily sees that ${n \choose d-k+1} \leq n^{d-k+1}$ is a trivial upper bound on $\deg_k(\xi_n)$. On the other hand, because there are ${n \choose k}$ simplices of dimension $(k-1)$ and $c(d,K) n^{d} (1+o(1))$ empty simplices, the typical degree is of order $ n^{d-k}$, which is a lower bound for $\deg_k(\xi_n)$.
In contrast to the case  $k=d$ where the \emph{upper} bound gives the correct order, we are showing that for the case $k=1$  and $d \geq 3 $ the \emph{lower} bound gives indeed the correct asymptotic behavior. In the case of $k=1$ and $d=2$ we get an additional logarithmic term.
\begin{theorem}\label{deg_k}
Let  $\xi_n$ be a set of $n$ independent and uniformly chosen random points from a convex body $K\subset \R^d$. Then there are constants $c(d), c(d,K)$ such that
\begin{itemize}
	\item[(1)] for all $d>2$:
	$$ c(d) n^{d-1} \leq  \E \deg_1 \xi_n \leq c(d,K) n^{d-1}, $$
	\item [(2)] for $d=2$:
	$$ c n \leq  \E \deg_1 \xi_n \leq c(K) n\ln^\frac 12 n .$$
\end{itemize}
\end{theorem}
It is unclear for us whether the logarithmic factor for $d=2$ is an artifact of our method of proof or reflects on special behavior in planar geometry.

We believe that the basic proof idea of Theorem \ref{deg_k} works for any $k=1,\ldots,d-1$. Note that going through the steps of the proof one sees that the cases $k=2,\ldots,d-1$ get computationally much more involved and intricate and we have not been able to prove these cases. However, the following conjecture stands to reason:

\begin{conjecture}\label{conjecture}
Let  $\xi_n$ be a set of $n$ independent and uniformly chosen random points from a convex body $K\subset \R^d$. Then for $k=1, \dots d-1$ 
	\[
		\E \deg_k (\xi_n) = \Theta(n^{d-k})
	\]
	as $n \to \infty$.
\end{conjecture}

\section{Proof of Theorem \ref{th:deg_d}}

Assume $X \subset \R^d$ is a set of $n$ points in general position. For a $(d-1)$-simplex $[x_1,\ldots,x_{d}]$ we denote by $\cF_1([x_{1},\ldots,x_{d}])$ the set of edges of $[x_{1},\ldots,x_{d}]$. We write $| [x_{1},\ldots,x_{d}] | \leq a$ if and only if $\| e \| \leq a$ for every $e \in \cF_{1}([x_{1},\ldots,x_{d}])$. Next we define a functional that measures "closeness" of points as was done in \cite{BMR13} and \cite{Te18}, respectively. However, in \cite{Te18} this was done by checking whether there exists a point in $\{x_{1},\ldots,x_{d}\}$ such that all remaining points are below a certain distance to it. Here, we will measure this closeness by checking if all the edges of the simplex $[x_{1},\ldots,x_{d}]$ are below a certain length,
\begin{equation}\label{eq:defN}
	N_{\g n}(X) = \sum_{ \{x_{1},\ldots,x_{d} \}  \in {X \choose d} } \1\left(| [x_{1},\ldots,x_{d}] | \leq (\g n)^{-\frac {1}{d-1}}\right),
\end{equation}
where $\g$ is a constant to be chosen later. The choice of $n^{-\frac{1}{d-1}}$ as the bound for the edge length will ensure that the expectation of $N_{\g n}$ converges to a constant as $n \to \infty$. In the following a suitable choice for $\g$ will be made such that the $N_{\g n}$ equals one with sufficiently high probability. The essence of the proof is to build the star of empty simplices above precisely this base-simplex and to show that the number of spikes of this star is of order $n$.

Note that the choice of the functional in \cite{Te18} and the one we made here coincide for the $2$-dimensional case in \cite{BMR13}. 
Similarly as in \cite{BMR13} and \cite{Te18} we also introduce a second functional which additionally weights the summands by the respective degree of that $d$-tuple of points, i.e.,
\[
F_{\g n}(X) = \sum_{ \{ x_{1},\ldots,x_{d} \}  \in {X \choose d}} \1\left(| [x_{1},\ldots,x_{d}] | \leq (\g n)^{-\frac {1}{d-1}}\right) \deg_d(x_{1},\ldots,x_{d};X),
\]
and make use of the equation
\[
	F_{\g n}(X) \leq N_{\g n}(X) \deg_d(X).
\]
In particular, we will need the slightly weaker  version
\begin{equation}\label{eq:lower_bound_deg_d}
	 \deg_d(X) \geq  
	 F_{\g n}(X) \1(N_{\g n}(X)=1)    .
\end{equation}

Now we have gathered all the necessary preliminary statements to prove Theorem \ref{th:deg_d}. We start with two lemmas that will be needed later on. 
First we deal with a set $\xi_n$ of $n$ independent and uniformly chosen random points from a compact set $C \subset R^d$ with $\l_d(C)>0$ and with boundary of Lebesgue measure zero.

\begin{lemma}\label{constN}
	Let $\xi_n$  be a set of  $n$ independent and uniformly chosen random points from $C$. Then, there exists a positive constant $c(d) \leq (d!)^{-1} \k_d^{d-1}$ (with equality in the case $d=2$) such that
	\[
	 \E N_{\g n}(\xi_n)  \leq 	\lim_{n \to \infty} \E N_{\g n}(\xi_n) = c(d) \g^{-d} \l_d(C)^{-(d-1)}.
	\]
\end{lemma}

\begin{proof}
	By the definition \eqref{eq:defN} we have
	\begin{align}
	\E N_{\g n}(\xi_n) 
	&= \label{eq:defintENP}
    \binom{n}{d}  \P\left(| [X_{1},\ldots,X_{d}] | \leq (\g n)^{-\frac {1}{d-1}} \right) 
	\\ &= \label{eq:defintEN}
	\frac{\binom{n}{d}}{ \l_d(C)^{d}} \int\limits_{C^{d}} \1\left(| [x_{1},\ldots,x_{d}] | \leq (\g n)^{-\frac {1}{d-1}} \right)\, \prod_{i=1}^{d} \dint x_i 
	\\ &= \nonumber
	\frac{\binom{n}{d}}{ \l_d(C)^{d}} \int\limits_{C^{d}} \1\left(\|e\| \leq (\g n)^{-\frac {1}{d-1}}\ , \forall e \in \cF_{1}([x_1,\ldots,x_d])\right)\, \prod_{i=1}^{d} x_i .
	\end{align}
	Now we transform $x_i = (\g n)^{-\frac {1}{d-1}}y_i + x_d$, $i=1, \dots, d-1$, and use the notation $C^\prime = (\g n)^{\frac {1}{d-1}} (C-x_d)$.
	\begin{align}\label{eq:ENnbound}
	\E N_{\g n}(\xi_n) &= \nonumber
	\frac{\binom{n}{d}}{ \l_d(C)^{d}}   (\g n)^{-d} 
	\\ & 
	\times \int\limits_{C} \int\limits_{( C^\prime )^{d-1}} \prod\limits_{1\leq i < j < d} \1(\|y_i-y_j\| \leq 1) \prod\limits_{1\leq i < d} \1(\|y_i\| \leq 1)\,	\prod_{i=1}^{d-1}\dint y_i\, \dint x_d 
	\end{align}
Because $C' \to \R^d$ if $x_d$ is in the interior of $C$, and the boundary of $C$ has measure zero, this shows that 
	\begin{align*}
	\lim\limits_{n \to \infty}\E N_{\g n}(\xi_n)
	&=
	\ (d!)^{-1} \g^{-d} \l_d(C)^{-(d-1)} \int\limits_{(B^d)^{d-1}} \prod\limits_{1\leq i < j < d} \1(\|y_i-y_j\| \leq 1)\, \prod_{i=1}^{d-1}\dint y_i 
	\\ & \leq 
	\ (d!)^{-1} \g^{-d} \l_d(C)^{-(d-1)} \k_d^{d-1}
	\end{align*}
which, clearly, also  is an upper bound for \eqref{eq:ENnbound}. Note that in the planar case this inequality is in fact an equality. 
\end{proof}

\begin{lemma}\label{ProbN}
	Let $\xi_n$  be a set of  $n$ independent and uniformly chosen random points from $C$. Then $N_{\g n}(\xi_n)$ converges to a Poisson random variable with mean $\E N_{\g n}(\xi_n)$, in particular
    \[
		\P( N_{\g n}(\xi_n)=0) = e^{- \E N_{\g n}(\xi_n)} (1+o(1)) .
	\]
\end{lemma}

\begin{proof}
We use the Poisson convergence theorem by Barbour and Eagleson \cite{BarbourEagleson} for dissociated random variables. To use the notations from their paper, we have $N_{\g n} (\xi_n) = \sum_{J \in {\xi_n \choose d}} X_J$ with 
$$ X_J = 
\1\left(\| e\| \leq (\g n)^{-\frac {1}{d-1}},\, \forall e \in \cF_1([J]) \right) 
= \prod\limits_{e \in \cF_1([J])}\1\left(\| e\| \leq (\g n)^{-\frac {1}{d-1}}\right) 
 $$  
These random variables are dissociated since $X_J$ and $X_K$ are independent if $J$ and $K$ are disjoint. 
Note that we have proved above that $p_J = \E X_J = O(n^{-d})$. Assume that $J, K \in {\xi_n \choose d}$ have $r$ points in common, i.e. $J=\{X_1, \dots X_d\},\, K=\{X_{d+1-r}, \dots, X_{2d-r} \}$. Then as in the proof of Lemma \ref{constN} above and using the same substitution we have 
	\begin{align*}
	\E X_J X_K
	&= 
	\frac{1}{ \l_d(C)^{2d-r}} \int\limits_{C^{2d-r}} 
	\prod\limits_{e \in \cF_1([J]) \cup \cF_1([K])}\1\left(\| e\| \leq (\g n)^{-\frac {1}{d-1}}\right)  \prod_{i=1}^{2d-r} x_i .
	\\ & =
	\frac{1}{ \l_d(C)^{2d-r}}   (\g n)^{- \frac{d(2d-r-1)}{d-1}} 
	\times \int\limits_{C} \int\limits_{( C^\prime )^{2d-r-1}} \prod\limits_{1\leq i < j < d} \1(\|y_i-y_j\| \leq 1) 
	\\ & \hskip4cm \times 
	\prod\limits_{d+1-r\leq i < j < 2d-r} \1(\|y_i-y_j\| \leq 1) 
	\prod \1(\|y_i\| \leq 1)   \prod \dint y_i\, \dint x_d 
	\end{align*}
which proves that $\E X_J X_K$ is of order $O(n^{- \frac{d(2d-r-1)}{d-1}})$.

Now, Theorem 2 of \cite{BarbourEagleson} says that for $A \subset \N_0$ and a Poisson distributed random variable $Z$ with mean $\E N_{\g n} (\xi_n)$,
\begin{align*}
|\P(N_{\g n}(\xi_n) \in A) - \P(Z \in A)| 
&\leq 
\frac 1{\max(1, N_{\g n}(\xi_n))} \sum_{J \in {\xi_n \choose d}} 
\! \left\{ p_J^2 + \sum _{r=1}^{d-1} \sum_{\substack{K \in {\xi_n \choose d}\\ |J \cap K|=r }} \!(p_J p_K + \E X_J X_K ) \! \right\}   
\\ &\leq 
c(d,K) n^d 
\left\{ n^{-2d} + \sum _{r=1}^{d-1} n^{d-r} (n^{-2d} + n^{- \frac{d(2d-r-1)}{d-1}} ) \right\}   
\\ &
= O(n^{- \frac 1{d-1}}).
\end{align*}
In the planar case a similar result was already known by work of Brown and Silverman \cite{BS1, BS2}.
\end{proof}

We return now to the proof of Theorem \ref{th:deg_d}. So let $\xi_n$ be a set of $n$ independent and uniformly chosen random points from a convex body $K\subset \R^d$. We plug $\xi_n$ into \eqref{eq:lower_bound_deg_d} and take expectations on both sides. This leads to
\begin{align*}
	&\E \deg_d(\xi_n)  \geq \E \left[ F_{\g n}(\xi_n) \1(N_{\g n}(\xi_n) = 1) \right] \\
	& \geq \frac{\binom{n}{d}}{\l_d(K)^d}  \int\limits_{K^d} \1\left( | [x_1,\ldots,x_d] |\leq (\g  n)^{- \frac {1}{d-1}} \right) \E \left[\deg_d(x_1,\ldots,x_d;\xi^\prime_n) \1(N_{\g n}(\xi^\prime_n)=1)\right]\, \prod_{i=1}^d\dint x_i,
\end{align*}
where $\xi^\prime_n=\xi_{n-d} \cup \{x_1,\ldots,x_d\}$. 
Next we need a lower bound for the expectation. We have
\begin{align}\label{eq:inproofdeg}
\E &\left[\deg_d\, (x_1,\ldots,x_d;\xi^\prime_n)  \1(N_{\g n}(\xi^\prime_n)=1)\right] 
\nonumber \\ & \nonumber 
\qquad= \E \sum_{X_{d+1} \in \xi_{n-d}} \1( [x_1,\ldots,x_d,X_{d+1}]^o \cap \xi_n^\prime = \emptyset ) \1(N_{\g n}(\xi_n^\prime)=1) 
\\ & 
\qquad= (n-d) \max\limits_{x_1, \dots, x_{d+1} \in K} \P([x_1,\ldots,x_d,x_{d+1}]^o \cap \xi_n^{\prime\prime} = \emptyset,\ N_{\g n}(\xi_n^{\prime\prime})=1 ) 
\end{align}
where $\xi_n^{\prime\prime}=\xi_{n-d-1} \cup \{x_1,\ldots,x_{d+1}\}$. 
Now set 
$$A(x_1,\ldots,x_{d+1})=\left([x_1,\ldots,x_{d+1}]  \cup \bigcup_{i=1}^{d+1} B\left(x_i, (\g  n)^{-\frac 1{d-1}}\right)\right) \cap K . $$ 
The base $[x_1, \ldots, x_d]$ of the simplex $[x_1,\ldots, x_{d+1}]$ has edge length bounded by $(\g n)^{-\frac d{d-1}} $ and thus is contained in 
$B\left(x_1, (\g n)^{-\frac 1{d-1}}\right)$. Its height is bounded by $D(K)$, the diameter of $K$. This implies
$$ \l_{d}(A(x_1,\ldots,x_{d+1})) \leq 
\frac 1d \kappa_d (\g n)^{-\frac {d-1}{d-1}} D(K) +  
(d+1)  \kappa_d (\g n)^{-\frac d{d-1}} 
=
\frac 1d \kappa_d (\g n)^{-1} D(K) (1+o(1))
$$
as $n \to \infty$. Hence, the probability that no point of $\xi_{n-d-1}$ is contained in $A(x_1, \ldots, x_{d+1})$ can be computed via
\begin{align*}
\P(  \xi_{n-d-1} \cap A (x_1,\ldots,x_{d+1})^o = \emptyset) 
&= 
 \left(1-  \frac{\l_d(A (x_1,\ldots,x_{d+1}))}{\l_d(K)} \right)^{n-d-1} 
\\ & \geq
e^{-  \frac 1d \kappa_d \g^{-1} D(K) \l_d(K)^{-1} } (1+o(1)) 
\end{align*}
as $n \to \infty$.
Note that 
$ N_{\g n}(\xi_{n-d-1})=0  $ and $\xi_{n-d-1} \cap A (x_1,\ldots,x_{d+1})^o = \emptyset$ 
imply 
$ N_{\g n}(\xi_n^{\prime\prime})=1  $ and $[x_1,\ldots,x_d,x_{d+1}]^o \cap \xi_n^{\prime\prime} = \emptyset$.
Hence
\begin{align*}
\P( & [x_1,\ldots,x_d,x_{d+1}]^o \cap \xi_n^{\prime\prime} = \emptyset,\ N_{\g n}(\xi_n^{\prime\prime})=1 ) 
\\ & \geq 
\P(N_{\g n}(\xi_{n-d-1})=0 \,\vert\, \xi_{n-d-1} \cap A (x_1,\ldots,x_{d+1})^o = \emptyset) 
e^{-  \frac 1d \kappa_d \g^{-1} D(K) \l_d(K)^{-1} } (1+o(1)) 
\end{align*}
and since in this conditional probability the random points are chosen uniformly in the compact set $K \backslash  A (x_1,\ldots,x_{d+1})^o$ of volume smaller than $\l_d(K)$, we can use Lemma \ref{ProbN} to obtain
\begin{align*}
\P( & [x_1,\ldots,x_d,x_{d+1}]^o \cap \xi_n^{\prime\prime} = \emptyset,\ N_{\g n}(\xi_n^{\prime\prime})=1 ) 
\\ & \geq 
\P(N_{\g n}(\xi_{n-d-1})=0 \,\vert\, \xi_{n-d-1} \cap A (x_1,\ldots,x_{d+1})^o = \emptyset) 
e^{-  \frac 1d \kappa_d \g^{-1} D(K) \l_d(K)^{-1} } (1+o(1)) 
\\ & \geq 
e^{-  c(d) \g^{-d} \l_d(K)^{-(d-1)} } 
e^{-  \frac 1d \kappa_d \g^{-1} D(K) \l_d(K)^{-1} } (1+o(1)) 
.
\end{align*}
We plug this into \eqref{eq:inproofdeg} and get 
$$ 
\E [\deg_d\, (x_1,\ldots,x_d;\xi^\prime_n)  \1(N_{\g' n}(\xi^\prime_n)=1)] 
\geq
n e^{-  c(d) \g^{-d} \l_d(K)^{-(d-1)} } 
e^{-  \frac 1d \kappa_d \g^{-1} D(K) \l_d(K)^{-1} } (1+o(1))  
 $$
as $n \to \infty$. 
Since  this is independent of $x_1, \dots , x_d$, we can conclude using \eqref{eq:defintEN} that 
\begin{align*}
\E \deg_d(\xi_n) 
\geq &\  
n e^{-  c(d) \g^{-d} \l_d(K)^{-(d-1)} }
e^{-  \frac 1d \kappa_d \g^{-1} D(K) \l_d(K)^{-1} } (1+o(1)) 
\\ &\ \times 
\frac{\binom{n}{d}}{\l_d(K)^d} \int\limits_{K^d} \1\left( | [x_1,\ldots,x_d]  |\leq (\g^\prime n)^{- \frac 1{d-1}} \right) \,  \dint x_1 \ldots \dint x_d 
\\  & = 
n e^{-  c(d) \g^{-d} \l_d(K)^{-(d-1)} }
e^{-  \frac 1d \kappa_d \g^{-1} D(K) \l_d(K)^{-1} }
\E N_{\g n} (\xi_n)  \  (1+o(1))  
\\  & = 
n e^{-  c(d) \g^{-d} \l_d(K)^{-(d-1)} }
e^{-  \frac 1d \kappa_d \g^{-1} D(K) \l_d(K)^{-1} }
c(d) \g^{-d} \l_d(K)^{-(d-1)} \  (1+o(1))  
\end{align*}
as $n \to \infty$, where we used the limit provided in Lemma \ref{constN}. There exists some $\g' $ which maximizes the right hand side of the inequality, and the maximum clearly is positive. This yields Theorem \ref{th:deg_d}. 

In the planar case, the maximum can be computed explicitly. For the unit circle it is attained at $\g = 1$, which yields
\begin{equation}\label{eq:constd=2}
\E \deg_2(\xi_n) 
\geq \frac 12 e^{- \frac 32} \,  n  \  (1+o(1)).  
\end{equation}

\qed

\section{Proof of Theorem \ref{th:emptytr} and Theorem \ref{deg_k}}
\subsection{Preliminaries}
For the proofs of Theorem \ref{th:emptytr} and Theorem \ref{deg_k} several lemmas will be needed. The first one is a quite general bound on the maximum of a collection of random variables and is a straightforward modification of a result of Aven \cite[Lemma 2.2]{Aven85}. Note that neither independence nor identical distributions are required for this lemma.

\begin{lemma}\label{Aven}
	Let $I$ be a finite index set and $S_i, C_i$, $i \in I$, be random variables. Then, for any $p\geq 1$
	\[
		\E \max_{i \in I} S_i \leq \E \max_{i \in I} C_i + \left( \sum_{i \in I} \E | S_i - C_i |^p \right)^\frac{1}{p}.
	\]
\end{lemma}

The next lemma is a slight generalization due to Reitzner \cite{survMR} of a result of Rhee and Talagrand \cite{RheeTal} which will prove to be a very practical tool when it comes to simplifying and bounding the second summand in Aven's Lemma.
\begin{lemma}\label{Rhee}
	Let $S(X_1,\ldots,X_m)$ be a real symmetric function of i.i.d. random vectors $X_i$, $1\leq i \leq m+1$. Then, for $p \geq 1$, 
\begin{equation*}
 \E \left|S-\E S \right|^p 
\leq m^{\frac p2} c_p \E   |S-S'|^p 
\end{equation*}
for any real symmetric function $S'(X_1,\ldots,X_{m+1})$, with $c_p= 2^p (18 \sqrt q \ p)^p$ where $ \frac 1p + \frac 1q=1$.
\end{lemma}

And lastly we will need two versions of the affine Blaschke-Petkantschin formula. Here, we denote by $\cH_d$ the affine Grassmannian of $(d-1)$-dimensional affine hyperplanes of $\R^d$ equipped with the unique rigid motion invariant Haar measure $\mu_{d-1}$, normalized by $\mu_{d-1}(\{ H \in \cH_d : H \cap {B}^d \neq \emptyset \}) = 2$, where ${B}^d$ denotes the Euclidean unit ball in $\R^d$ and $\kappa_{d}=\l_{d}({B}^{d})$. First, we state the classical Blaschke-Petkantschin formula, see e.g. the book by Schneider and Weil \cite[Theorem 7.2.7.]{SW}. In the following $\dint H$ denotes integration with respect to $\mu_{d-1}$, and $\dint x_i$ integration with respect to the Lebesgue measure on the affine space given by the range of integration.
\begin{lemma}\label{lem:BPF}
Let $g: (\R^d)^{d} \to \R$ be a nonnegative measurable function. Then 
$$ \int\limits_{(\R^d)^{d}} g(x_1,\ldots,x_{d}) \prod_{i=1}^{d} \dint x_i 
= 
b_d \int\limits_{\cH_d} \int\limits_{H^d} g(x_1,\ldots,x_{d}) 
\l_{d-1}([x_1,\ldots,x_d]) \prod_{i=1}^{d} \dint x_i \ \dint H 
$$
with the constant 
$b_d = \frac{d \kappa_d}{2 } (d-1)!$. 
\end{lemma}
Second, we need the following version of the affine Blaschke-Pentkantschin formula, which was provided by Hug and Reitzner \cite{HuR05}.
\begin{lemma}\label{lem:HugReitzner}
Let $0 \leq k \leq d-1$ and let $g: (\R^d)^{2d-k} \to \R$ be a nonnegative measurable function. Then 
\begin{align*}
	&\int\limits_{(\R^d)^{2d-k}} g(x_1,\ldots,x_{2d-k}) \prod_{i=1}^{2d-k} \dint x_i 
	= b_{d}^2
	\int\limits_{(\cH_d)^2} \int\limits_{H_1^{d-k}} \int\limits_{(H_1 \cap H_2)^k} \int\limits_{H_2^{d-k}} g(x_1,\ldots,x_{2d-k}) \\
	& \times \l_{d-1}([x_1,\ldots,x_d]) \l_{d-1}([x_{d-k+1},\ldots,x_{2d-k}]) (\sin \varphi)^{-k} \prod_{i=d+1}^{2d-k} \dint x_i \prod_{i=d-k+1}^{d} \dint x_i  \prod_{i=1}^{d-k} \dint x_i \prod_{i=1}^2 \dint H_i,
\end{align*}
where  $\varphi$ denotes the angle between the normal vectors of $H_1$ and $H_2$, and the constant $b_{d}$ is chosen as in Lemma \ref{lem:BPF}.
\end{lemma}

\subsection{Proof of Theorem \ref{th:emptytr}: the number of empty simplices}
Let  $\xi_n$ be a set of $n$ independent and uniformly chosen random points from a convex body $K\subset \R^d$. Recall that by $N_\tr (\xi_n)$ we mean the total number of empty simplices, i.e. 
$$
N_\tr (\xi_n)= \sum_{\{ X_1, \dots, X_{d+1}\} \in { \xi_n \choose d+1}} \1 ([X_{1},\ldots,X_{d+1}]^o \cap \xi_{n}  = \emptyset) 
. $$
Given $x_1, \dots, x_{d+1}$, the probability that a uniform point avoids the convex hull 
$[x_1, \dots, x_{d+1}]$ is given by 
$$1-\frac{\l_d[x_1, \dots, x_{d+1}]}{\l_d(K)}  .$$
Hence, taking expectations and using that the points are identically distributed gives
\begin{eqnarray*}
\E N_\tr (\xi_n)
 &=&
{n \choose d+1} \l_d(K)^{-(d+1)} \int \limits_{K^{d+1}} \left(1-\frac{\l_d[x_1, \dots, x_{d+1}]}{\l_d(K)}\right)^{n-d} \, \dint x_1 \dots \dint x_{d+1}
\\ &=&
{n \choose d+1} \l_d(K)^{-(d+1)} \int \limits_{K^{d+1}} \left(1-  \frac{\l_{d-1}[x_1, \dots, x_{d}]}{d \l_d(K)}\, |h_{d+1}| \right)^{n-d} \, \dint x_1 \dots \dint x_{d+1}
\end{eqnarray*}
where $h_{d+1}$ is the (signed) distance of $x_{d+1}$ to the affine hull of $x_1, \dots, x_d$. 
Denote by $H=H(0)$ the affine hull of $x_1, \dots, x_d$, and by $H(h_{d+1})$ the parallel hyperplane through $x_{d+1} = (y_{d+1}, h_{d+1})$ for $h_{d+1} \in K\vert_{H^\perp}$ and $y_{d+1} \in K \cap H(h_{d+1})$. 
By Fubini's theorem we have
\begin{eqnarray*}
\E N_\tr (\xi_n)
&=& 
{n \choose d+1}  \l_d(K)^{-(d+1)} 
\\ &&
\int \limits_{K^{d}} \int\limits_{K\vert_{H^\perp}}\left(1-  \frac{\l_{d-1}[x_1, \dots, x_{d}]}{d \l_d(K)} \, |h_{d+1}| \right)^{n-d} \l_{d-1}(K \cap H(h_{d+1}))\, \dint h_{d+1}\dint x_1 \dots \dint x_d 
\end{eqnarray*}
We substitute $h=n h_{d+1}$ and obtain
\begin{eqnarray*}
\E N_\tr (\xi_n)
 &=&
\frac 1n {n \choose d+1}  \l_d(K)^{-(d+1)} 
\\&&
\times \int \limits_{K^{d}} \int\limits_{n K\vert_{H^\perp}}\left(1-  \frac{\l_{d-1}[x_1, \dots, x_{d}]}{d \l_d(K) n } \, |h| \right)^{n-d} \l_{d-1} \left(K \cap H\left(\frac hn \right)\right)\, \dint h\dint x_1 \dots \dint x_d .
\end{eqnarray*}
Lebesgue's dominated convergence theorem shows that 
\begin{align*}
\lim_{n \to \infty} n^{-d} \E N_\tr (\xi_n)
& =
\frac 1{(d+1)!}  \l_d(K)^{-(d+1)}\int \limits_{K^{d}} \int\limits_{\R}e^{-  \frac{\l_{d-1}[x_1, \dots, x_{d}]}{d \l_d(K) } \, |h| } \l_{d-1} \left(K \cap H\left(0\right)\right)\, \dint h\dint x_1 \dots \dint x_d 
\\ & =
\frac 2{(d+1)!}  \l_d(K)^{-(d+1)}\int \limits_{K^{d}} \frac{d \l_d(K) }{\l_{d-1}[x_1, \dots, x_{d}]}  \l_{d-1} \left(K \cap H\right)\, \dint x_1 \dots \dint x_d .
\end{align*}
Using the Blaschke-Petkantschin formula from Lemma \eqref{lem:BPF} we obtain
\begin{eqnarray*}
\lim_{n \to \infty} n^{-d} \E N_\tr (\xi_n)
 &=&
\frac {2d  }{(d+1)!} b_d  \l_d(K)^{-d} \int\limits_{\cH_d} \int \limits_{(K\cap H)^{d}}   \l_{d-1} \left(K \cap H\right)\, \dint x_1 \dots \dint x_d \dint H
\\ &=&
\frac {d \k_d }{(d+1)}  \l_d(K)^{-d} \int\limits_{\cH_d} \l_{d-1} \left(K \cap H\right)^{d+1}\, \dint H .
\end{eqnarray*}

The inequality 
$$
\int_{\cH_d} \l_{d-1} (K \cap H) ^{d+1} \dint H  
	\leq \frac {\k_{d-1}^{d+1}  \k_{d^2}}{\k_d^d  \k_{(d-1)(d+1)}} \l_d(K)^{d} 
$$
is classical, see e.g. \cite[(8.56)]{SW}, with equality if $d=2$, and for $d \geq 3$ iff $K$ is an ellipsoid. This implies for $d \geq 3$
\begin{eqnarray*}
\lim_{n \to \infty} n^{-d} \E N_\tr (\xi_n)
&\leq &
\frac {d  }{(d+1)}  \frac {\k_{d-1}^{d+1}  \k_{d^2}}{\k_d^{d-1}  \k_{(d-1)(d+1)}} ,
\end{eqnarray*}
with equality  if $K$ is an ellispoid, and for $d=2$ 
$$
\lim_{n \to \infty} n^{-2} \E N_\tr (\xi_n)
= 2 .
$$
To get a lower bound, observe that each $(d-1)$-dimensional simplex formed by points of $\xi_n$ can either be complemented to at least two empty simplices by the two points which are closest to its affine hull on each side of the corresponding hyperplane, if the simplex is not a facet of the the convex hull of $\xi_n$, or to at least one empty simplex if it is a facet of the convex hull of $\xi_n$. Since the number of facets is much smaller than $n^{d}$, we have 
$$ \E N_\tr(\xi_n) \geq 2 {n \choose d} (1+o(1)) = \frac 2{d!}\, n^{d} (1+o(1)) .
$$

\qed

\subsection{Proof of Theorem \ref{deg_k}: the $k$-degree}

Again we will be working with a random point set $\xi_n$ consisting of $n$ uniformly and independently chosen points from a convex body $K \subset \R^n$. As pointed out in the introduction the considerations for $\E N_\tr (\xi_n)$ show that 
\begin{equation}\label{eq:upplowbound}
	c(d,K) n^{d-k} \leq \E \deg_k(\xi_n) \leq c(d) n^{d-k+1}.
\end{equation}

\bigskip
We go on showing that for the case $k=1$ the lower bound is indeed the correct one. We will start by approaching the problem for general $k$ and specialize to $k=1$ in the second part of the proof. We invoke Lemma \ref{Aven}, and use that $\deg_k(X_{1},\ldots,X_{k};\xi_n) $ is identically distributed for all sets $\{ X_1, \dots, X_k \} \in { \xi_n \choose k}$. 
\begin{align}\label{eq:Aven-general}
	\E \deg_k(\xi_n) 
	&= 
	\E \max_{ \{ X_1,\ldots,X_k \} \in {\xi_n \choose k} } \deg_k\left(X_{1},\ldots,X_{k};\xi_n\right) 
	\\ & \leq \nonumber 
	\E \max_{  \{ X_1,\ldots,X_k \} \in {\xi_n \choose k} }  \E \left(\deg_k(X_{1},\ldots,X_{k};\xi_n)|X_{1}, \dots, X_{k}\right) 
	\\ & +\nonumber 
	\left( {n \choose k} \E \Big| \deg_k\left(X_{1},\ldots,X_{k};\xi_n\right) - \E \left(\deg_k(X_{1},\ldots,X_{k};\xi_n)|X_1, \dots, X_k\right) \Big|^{p}  \right)^\frac{1}{p}
\end{align}

Next, we transform the second summand into something that can be handled properly. Let $X_{n+1}$ be an additional uniform random point from $K$ independent of $\xi_n$, and write $\xi_{n+1} = \xi_n \cup \{ X_{n+1} \}$. We apply Lemma \ref{Rhee} to $\xi_n \setminus \{X_{1},\ldots,X_k\}$, for $\{X_1,\ldots,X_k\}$ fixed, and obtain
\begin{align*}
	\E &\Big| \deg_k(X_{1},\ldots,X_{k};\xi_n) - \E  (\deg_k(X_{1},\ldots,X_{k};\xi_n)|X_1, \dots, X_k) \Big|^{p} 
	\\ &\leq 
	c_{p,k} (n-k)^{\frac p2}  \E \Big| \deg_k(X_{1},\ldots,X_{k};\xi_n) - \deg_k(X_{1},\ldots,X_{k};\xi_{n+1}) \Big|^{p} 
	\\ & \leq 
	c_{p,k} n^{\frac p2} \E\, \vast| \sum_{ \{ X_{{k+1}},\ldots,X_{{d+1}} \} \in {\xi_n \setminus \{X_{1},\ldots,X_k\} \choose d-k+1}} \1([X_{1},\ldots,X_{{d+1}}]^o \cap \xi_n = \emptyset)
	\\ &
	- \sum_{ \{ X_{{k+1}},\ldots,X_{{d+1}} \} \in {\xi_{n+1} \setminus \{X_{1},\ldots,X_k\} \choose d-k+1}} \1([X_{1},\ldots,X_{i_{d+1}}]^o \cap \xi_{n+1}  = \emptyset) \vast|^{p} ,
\end{align*}
with $c_{p,k}$ being a generic positive constant only depending on $p$ and $k$. 
These two sums can now be decomposed. The first sum can be decomposed in those simplices which contain the additional point $X_{n+1}$ in their interior and those which do not, i.e.,
\begin{align*}
	\sum_{ \{ X_{{k+1}},\ldots,X_{{d+1}} \} \in {\xi_{n} \setminus \{X_{1},\ldots,X_k\} \choose d-k+1}}
	&
	\1([X_{1},\ldots,X_{{d+1}}]^o \cap \xi_n = \emptyset)
	\\ & = 
	\sum_{ \{ X_{{k+1}},\ldots,X_{{d+1}} \} \in {\xi_{n} \setminus \{X_{1},\ldots,X_k\} \choose d-k+1}} \1([X_{1},\ldots,X_{{d+1}}]^o \cap \xi_{n+1}  = \emptyset) 
	\\ & + 
	\sum_{ \{ X_{{k+1}},\ldots,X_{{d+1}} \} \in {\xi_{n} \setminus \{X_{1},\ldots,X_k\} \choose d-k+1}} \1([X_{1},\ldots,X_{{d+1}}]^o \cap \xi_{n+1}  = \{X_{n+1}\}).
\end{align*}
Analogously, we can decompose the second sum by distinguishing those simplices that contain the additional point $X_{n+1}$ as a vertex and those which do not, i.e.,
\begin{align*}
	\sum_{ \{ X_{{k+1}},\ldots,X_{{d+1}} \} \in {\xi_{n+1} \setminus \{X_{1},\ldots,X_k\} \choose d-k+1}} 
	&
	\1([X_{1},\ldots,X_{{d+1}}]^o \cap \xi_{n+1}  = \emptyset)
	\\ & = 
	\sum_{ \{ X_{{k+1}},\ldots,X_{{d}} \} \in {\xi_{n} \setminus \{X_{1},\ldots,X_k\} \choose d-k}} 
	\1([X_{1},\ldots,X_{{d}}, X_{n+1}]^o \cap \xi_{n+1}  = \emptyset)
	\\ & + 
	\sum_{ \{ X_{{k+1}},\ldots,X_{{d+1}} \} \in {\xi_{n} \setminus \{X_{1},\ldots,X_k\} \choose d-k+1}} 
	\1([X_{1},\ldots,X_{{d+1}}]^o \cap \xi_{n+1}  = \emptyset).
\end{align*}
Plugging these decompositions back into the formula yields
\begin{align*}
	\E &\Big| \deg_k(X_{1},\ldots,X_{k};\xi_n) - \E \deg_k(X_{1},\ldots,X_{k};\xi_n|X_1, \dots, X_k) \Big|^{p} 
	\\ &\leq 
	c_{p,k} n^{\frac p2} \E \vast| 
	\sum_{ \{ X_{{k+1}},\ldots,X_{{d+1}} \} \in {\xi_{n} \setminus \{X_{1},\ldots,X_k\} \choose d-k+1}} \1([X_{1},\ldots,X_{{d+1}}]^o \cap \xi_{n+1}  = \{X_{n+1}\})
	\\ & \hskip2cm  - 
	\sum_{ \{ X_{{k+1}},\ldots,X_{{d}} \} \in {\xi_{n} \setminus \{X_{1},\ldots,X_k\} \choose d-k}} 
	\1([X_{1},\ldots,X_{{d}}, X_{n+1}]^o \cap \xi_{n+1}  = \emptyset)
	\vast|^{p} .
\end{align*}
Note, that each simplex that fulfills $[X_{1},\ldots,X_{{d+1}}]^o \cap \xi_{n+1}  = \{X_{n+1}\}$ gives rise to $d-k+1$ empty simplices, i.e., meaning that $[\{X_{1},\ldots, X_{{d+1}} \} \setminus \{X_{j} \},X_{n+1}]^o \cap \xi_{n+1}=\emptyset$ holds, for every $j=k+1,\ldots,d+1$. Hence, we have
\begin{align*}
	&\sum_{ \{ X_{{k+1}},\ldots,X_{{d}} \} \in {\xi_{n} \setminus \{X_{1},\ldots,X_k\} \choose d-k}} \1([X_{1},\ldots,X_{{d}}, X_{n+1}]^o \cap \xi_{n+1}  = \emptyset)
	\\& \hskip2cm \geq 
	(d-k+1)
	\sum_{ \{ X_{{k+1}},\ldots,X_{{d+1}} \} \in {\xi_{n} \setminus \{X_{1},\ldots,X_k\} \choose d-k+1}} 
	\1([X_{1},\ldots,X_{{d+1}}]^o \cap \xi_{n+1}  = \{X_{n+1}\})
	\end{align*}
which means that the second sum dominates the first one. From this we arrive at
\begin{align*}
	\E 
	& 
	\left| \deg_k(X_{1},\ldots,X_{k};\xi_n) - \E \deg_k(X_{1},\ldots,X_{k};\xi_n|X_1, \dots, X_k) \right|^{p} 
	\\ &\leq 
	c_{p,k} n^{\frac p2} \E \left( 
	\sum_{ \{ X_{{k+1}},\ldots,X_{{d}} \} \in {\xi_{n} \setminus \{X_{1},\ldots,X_k\} \choose d-k}} 
	\1([X_{1},\ldots,X_{{d}}, X_{n+1}]^o \cap \xi_{n+1}  = \emptyset)
	\right)^{p} .
\end{align*}
We relabel the random points $X_k$ in the sum, put this into \eqref{eq:Aven-general} and obtain 
\begin{align}\label{eq:Bound}
	\E 
	& 
	\deg_k(\xi_n) 
	\leq 
	\E \max_{ \{ X_1,\ldots,X_k \} \in { \xi_n \choose k} }  \E \left(\deg_k(X_{1},\ldots,X_{k};\xi_n)|X_{1}, \dots, X_{k}\right) 
	\\ & +\nonumber 
	c_{p,k} n^{\frac kp + \frac 12} \left( \E \left( 
	\sum_{ \{ X_{{1}},\ldots,X_{{d-k}} \} \in {\xi_{n-k} \choose d-k}} 
	\1([X_{1},\ldots,X_{{d-k}}, X_{n-k+1} , \dots,  X_{n+1}]^o \cap \xi_{n+1}  = \emptyset)
	\right)^{p}  \right)^\frac{1}{p}
\end{align}
We will go on by bounding the maximal expectation, and the number of additional simplices with $X_{n+1}$, in the case $k=1$. The appropriate choice of $p$ will turn out to be $2$. 
We are convinced that in principle this approach could be successful for all $k$, albeit with a different $p$ than $2$, but we have not been able to obtain precise bounds in the cases $k=2, \dots, d-1$.

\subsubsection{The maximal expectation}

Given $X_1, \dots, X_{d+1}$, the probability that the simplex  $[X_1, \dots, X_{d+1}]$ is empty, is given by
$$ \left(1-\frac{\l_d([X_1, \dots, X_{d+1}])}{\l_d(K)} \right)^{n-d-1}. $$
We condition on  $X_1$ and apply the classical Blaschke-Petkantschin formula from Lemma \ref{lem:BPF}. 
\begin{align*}
& 
\P  ([X_{1},\ldots,X_{{d+1}}]^o \cap \xi_n = \emptyset | X_1)
\\ &=
\l_d(K)^{-d}\int\limits_{K^d}  \left(1-\frac{\l_d([X_1, x_2, \dots, x_{d+1}])}{\l_d(K)} \right)^{n-d-1} \, \prod_{i=2}^{d+1} \dint x_i
\\ &=
b_d \l_d(K)^{-d} \int\limits_{\cH_d} \int\limits_{(K \cap H)^d}  \left(1-\frac{\l_d([X_1, x_2,\dots, x_{d+1}])}{\l_d(K)} \right)^{n-d-1} \l_{d-1}([x_2, \dots, x_{d+1}])\prod_{i=2}^{d+1} \dint x_i\, \dint H
\end{align*}
Denote by $h_1$ the distance of $X_1$ to the hyperplane $H$. Then the volume of the simplex is given by the area $\l_{d-1}([x_2,\dots, x_{d+1}]) $ and its height $h_1$.
\begin{align*}
\P 
& 
([X_{1},\ldots,X_{{d+1}}]^o \cap \xi_n = \emptyset | X_1)
\\ &=
c(d,K)\int\limits_{\cH_d}  \int\limits_{(K \cap H)^d}  \left(1-\frac{h_1 \l_{d-1}([x_2,\dots, x_{d+1}])}{d \l_d(K)} \right)^{n-d-1} \l_{d-1}([x_2, \dots, x_{d+1}]) \prod_{i=2}^{d+1} \dint x_i\, \dint H
\end{align*}
We need an estimate for this from above. To simplify our notations we identify w.l.o.g. $X_1$ with the origin and thus the distance $h$ of the origin to the hyperplane $H$ equals the distance $h_1$ of $X_1$ to $H$. Then we use the bound $1-x \leq e^{-x}$. 
The integration with respect to the Haar measure on $\cH_d$ is given by the integration for $h $ with respect to Lebesgue measure, and integration with respect to unit normal vector $u$ of $H$ with respect to the Lebesgue measure on $S^{d-1}$. Furthermore, assume that each section $K \cap H$ is contained in a ball $R B^H \subset H$ of radius $R$. Using  Fubini's Theorem gives
\begin{align*}
\P 
& 
([X_{1},\ldots,X_{{d+1}}]^o \cap \xi_n = \emptyset | X_1)
\\ & \leq
c(d,K) \int\limits_{\cH_d}  \int\limits_{(R B^{H})^d}  e^{-\frac{ (n-d+1)\l_{d-1}([x_2,\dots, x_{d+1}])}{d \l_d(K)} h}  \l_{d-1}([x_2,\dots, x_{d+1}])\, \prod_{i=2}^{d+1} \dint x_i\, \dint H
\\ & =
c(d,K) \int\limits_{S^{d-1}} \int\limits_{(R B^{u^\perp})^d} \int\limits_0^\infty  e^{-\frac{ (n-d+1)\l_{d-1}([x_2,\dots, x_{d+1}])}{d \l_d(K)} h}  \l_{d-1}([x_2,\dots, x_{d+1}])\, \dint h\,  \prod_{i=2}^{d+1} \dint x_i\, \dint u  .
\end{align*}
Because the integrand is independent of rotations and the exponential function can easily be integrated we obtain
\begin{align}\label{eq:Pemptyn-1}
\P ([X_{1},\ldots,X_{{d+1}}]^o \cap \xi_n = \emptyset | X_1)
& \leq 
c(d,K) (n-d+1)^{-1} 
\end{align}
and thus we have a uniform bound which is independent of the choice of $X_1$. This immediately proves that there is a constant such that  
\begin{align}\label{eq:boundmaxE}
\E \max_{ i = 1,\ldots,n }  \E (\deg_1(X_{i};\xi_n)|X_{i}) 
 & = \nonumber
\E  \max_{ i = 1,\ldots,n }  \sum_{\{j_1, \dots, j_d\} \in {[n]\setminus\{i\} \choose d}} \P( [X_{i}, X_{j_1}, \dots , X_{j_d}]^o\cap \xi_n= \emptyset|X_{i}) 
\\ & \leq \nonumber
\E  \max_{ i = 1,\ldots,n }  {n-1 \choose d} 
c(d,K) (n-d+1)^{-1} 
\\ &\leq 
c(d,K) n^{d-1} .
\end{align}

\subsubsection{The number of additional simplices for the case $d\geq 3$}\label{sub34}

We need a bound for 
\begin{align*}
\E 
&
\left( 
\sum_{ \{ X_{1},\ldots,X_{d-1} \} \in {\xi_{n-1} \choose d-1}} 
\1([X_{1},\ldots,X_{{d-1}}, X_n, X_{n+1}]^o \cap \xi_{n+1}  = \emptyset)
\right)^{2}  
\\ & = 
\sum \P (([X_{i_1},\ldots,X_{i_{d-1}}, X_n, X_{n+1}]\cup [X_{j_1},\ldots,X_{j_{d-1}}, X_n, X_{n+1}])^o \cap \xi_{n+1}  = \emptyset)
\end{align*}
where the summation is over all pairs of $(d-1)$-tuples $ \{ i_{1},\ldots,i_{d-1} \} , \{ j_{1},\ldots,j_{d-1} \} \in {[n-1] \choose d-1} $.
These pairs will have $l$ points in common for $l=0,1, \dots, d-1$.  The number of pairs with $l$ common points can be counted by first choosing the $l$ common points, and then two disjoint sets for the remaining $(d-l-1)$ points. Using the multinomial coefficient this can be written as
$$
\sum_{ \{ i_{1},\ldots,i_{d-1} \} , \{ j_{1},\ldots,j_{d-1} \} \in {[n-1] \choose d-1} } 
\1( | \{ i_{1},\ldots,i_{d-1} \} \cap  \{ j_{1},\ldots,j_{d-1} \} |=l)
=
{n-1 \choose l,d-l-1,d-l-1}
$$
Since the multinomial coefficient can be estimated by $n^{2d-l-2}$, because the points are identically distributed and both simplices share in addition the two points $X_n, X_{n+1}$, this gives
\begin{align}\label{eq:sumP}
\E 
& \nonumber
\left( 
\sum_{ \{ X_{1},\ldots,X_{d-1} \} \in {\xi_{n-1} \choose d-1}} 
\1([X_{1},\ldots,X_{{d-1}}, X_n, X_{n+1}]^o \cap \xi_{n+1}  = \emptyset)
\right)^{2}  
\\ & \leq  
\sum_{l=0}^{d-1}  n^{2d-l-2} \P (([X_{1},\ldots,X_{d+1}]\cup [X_{d-l},\ldots,X_{2d-l}])^o \cap \xi_{n+1}  = \emptyset) 
.
\end{align}
Note that ultimately, as will become apparent in a moment, we would like to bound \eqref{eq:sumP} by $C(d,K)n^{2d-4}$. Thus, for $l=2,\ldots,d-1$ we can simply bound the probabilities in the respective summands by one. The bound
\begin{align}\label{trivial_bound}
\P
&
(([X_{1},\ldots,X_{d+1}]\cup [X_{d-1},\ldots,X_{2d-1}])^o \cap \xi_{n+1}  = \emptyset) \nonumber
\\ & \leq 
\P ([X_{1},\ldots,X_{d+1}]^o \cap \xi_{n+1}  = \emptyset) 
\\ & \leq 
C(d,K)n^{-1}, \nonumber
\end{align}
following from \eqref{eq:Pemptyn-1}, is a sufficient bound for the probability associated to the summand with $l=1$. However, the probability associated to the summand with $l=0$ needs to be bounded by $C(d,K)n^{-2}$ for us to be able to achieve our goal.

To do so, we transform the probability into a suitable integral, using that the points in $\xi_{n+1}$ are uniformly distributed in $K$.
\begin{align*}
\P 
&
(([X_{1},\ldots,X_{d+1}]\cup [X_{d},\ldots,X_{2d}])^o \cap \xi_{n+1}  = \emptyset) 
\\ & \leq  
\P (([X_{1},\ldots,X_{d+1}]\cup [X_{d},\ldots,X_{2d}])^o \cap \xi_{2d+1,n+1}  = \emptyset) 
\\ & = 
\l_d(K)^{-(2d)} \int\limits_{K^{2d}} \left( 1- \frac{\l_d([x_{1},\ldots,x_{d+1}]\cup [x_{d},\ldots,x_{2d}])}{\l_d(K)} \right)^{n-2d+1}\ \prod_{i=1}^{2d} \dint x_i 
\\ &\leq 
\l_d(K)^{-(2d)} \int\limits_{K^{2d}} e^{-(n-2d+1) \l_d([x_1,\ldots,x_{d+1}]\cup [x_{d},\ldots,x_{2d}]) \l_d(K)^{-1}} \prod_{i=1}^{2d} \dint x_i.
\end{align*}
We further dissect this expression with the help of the inequality $\l_d(A \cup B) \geq \frac{1}{2}(\l_d(A)+\l_d(B))$, for Lebesgue measurable sets $A$ and $B$. 
\begin{align*}
\P 
&
(([X_{1},\ldots,X_{d+1}]\cup [X_{d},\ldots,X_{2d}])^o \cap \xi_{n+1}  = \emptyset) 
\\ & \leq  
\l_d(K)^{-2d} \int\limits_{K^{2d}} e^{-\frac 12 (n-2d+1) (\l_d([x_1,\ldots,x_{d+1}])+\l_d( [x_{d},\ldots,x_{2d}])) \l_d(K)^{-1}} \prod_{i=1}^{2d} \dint x_i.
\end{align*}
We define $H_1=\aff(x_2,\ldots,x_{d+1})$ and $H_2=\aff(x_{d},\ldots,x_{2d-1})$, and parametrize $x_1=(b_1,h_1)$, $x_{2d+1}=(b_2,h_2)$ , with $b_i \in K|_{H_i}$, the projection of $K$ on $H_i$, and $h_i \in H_i^\perp$. As before, we make use of $\l_d([x_1,\ldots,x_{d+1}]) = d^{-1} \l_{d-1}([x_2,\ldots,x_{d+1}]) |h_1|$ as well as $\l_d([x_{d},\ldots,x_{2d}]) = d^{-1} \l_{d-1}([x_{d},\ldots,x_{2d-1}]) |h_2|$. Integrating out the inner integrals and bounding the exponential by one, then yields
\begin{align}\label{eq:P2simplices}
\P 
& \nonumber
(([X_{1},\ldots,X_{d+1}]\cup [X_{d},\ldots,X_{2d}])^o \cap \xi_{n+1}  = \emptyset) 
\\ & \leq  
\l_d(K)^{-2d} \int\limits_{K^{2d-2}} \left( \int\limits_K e^{-\frac 1{2d} (n-2d+1) h_1 \l_{d-1}([x_2,\ldots,x_{d+1}]) \l_d(K)^{-1}} \dint x_1 \right) 
\\ &  \nonumber \hskip4cm
\left( \int\limits_{K} e^{-\frac 1{2d} (n-2d+1) h_2\l_{d-1}( [x_{d},\ldots,x_{2d-1}]) \l_d(K)^{-1}} \dint x_{2d} \right) \prod_{i=2}^{2d-1} \dint x_i
\\ &  \nonumber \leq  
4 d^2  \l_d(K)^{-2(d-1)} \max_{H\in \cH_d} \l_{d-1}(K|_H)^2 \  (n-2d+1)^{-2}
\\ & \hskip4cm \nonumber
\int\limits_{K^{2d-2}}   \l_{d-1}([x_2,\ldots,x_{d+1}])^{-1}
\l_{d-1}( [x_{d},\ldots,x_{2d-1}])^{-1}\ \prod_{i=2}^{2d-1} \dint x_i.
\end{align}
We have  to show that the remaining integrals are finite.
We apply the affine Blaschke-Petkantschin formula of Lemma \ref{lem:HugReitzner}. Note that $H_1$ and $H_2$ are linked via the $2$ points $x_{d}$ and $x_{d+1}$.
\begin{align*}
\P 
&
(([X_{1},\ldots,X_{d+1}]\cup [X_{d},\ldots,X_{2d}])^o \cap \xi_{n+1}  = \emptyset) 
\\ & \leq  
c(d,K) n^{-2} 
\int\limits_{(\cH_d)^2} \int\limits_{(H_1 \cap K)^{d-2}} \int\limits_{(H_1 \cap H_2 \cap K)^2} \int\limits_{(H_2 \cap K)^{d-2}} 
(\sin \varphi)^{-2}  \prod_{j=d+2}^{2d-1} \dint x_j \prod_{j=d}^{d+1} \dint x_j \prod_{j=2}^{d-1} \dint x_j \prod_{i=1}^2 \dint H_i .
\end{align*}
where $\varphi$ denotes the angle between $H_1$ and $H_2$.
Since $K$ is bounded we obtain
\begin{align*}
\P 
&
(([X_{1},\ldots,X_{d+1}]\cup [X_{d},\ldots,X_{2d}])^o \cap \xi_{n+1}  = \emptyset) 
\\ & \leq  
c(d,K) n^{-2} 
\int\limits_{(\cH_d)^2} \1(H_1 \cap H_2 \cap K \neq \emptyset) (\sin \varphi)^{-2}  \prod_{i=1}^2 \dint H_i .
\end{align*}
In Lemma \ref{le:valuesin}, which is postponed to the appendix, we compute an upper bound for this integral which in particular shows that it is finite for $d \geq 3$. Hence continuing with Equation \eqref{eq:sumP} there is a constant such that 
$$
\E 
\left( 
\sum_{ \{ i_{1},\ldots,i_{d-1} \} \in {[n-1] \choose d-1}} 
\1([X_{i_1},\ldots,X_{i_{d-1}}, X_n, X_{n+1}]^o \cap \xi_{n+1}  = \emptyset)
\right)^{2}  
\leq c(d,K) n^{2d-4}	.
$$

Plugging this result, together with \eqref{eq:boundmaxE}, into \eqref{eq:Bound} yields
\begin{align}
\E \deg_{1}(\xi_n) \leq c(d,K) \left(n^{d-1} + n \left(n^{2d-4} \right)^{\frac 12}\right) \leq c(d,K) n^{d-1},
\end{align}
and, hence, concludes  the proof. \qed

\subsubsection{The number of additional simplices for the case $d=2$}

In the case $d=2$ we proceed with \eqref{eq:P2simplices}. Hence we need a bound for 
\begin{align*}
\P 
&
	(([X_{1},X_2,X_{3}]\cup [X_{2},X_3,X_{4}])^o \cap \xi_{n+1}  = \emptyset) 
	 \\ & \leq  
	\l_2(K)^{-4} \int\limits_{K^{2}} \left( \int\limits_K e^{-\frac 1{4} (n-3) h_1 \l_{1}([x_2,x_{3}]) \l_2(K)^{-1}} \dint x_1 \right) 
	 \\ & \hskip4cm
    \left( \int\limits_{K} e^{-\frac 1{4} (n-3) h_2\l_{1}( [x_{2},x_{3}]) \l_2(K)^{-1}} \dint x_{4} \right) \prod_{i=2}^{3} \dint x_i
	 \\ & \leq  
	16  \l_2(K)^{-2} \max_{H\in \cH_{1}^2} \l_{1}(K|_H)^2 \  (n-3)^{-2}
	 \\ & 
	 \int\limits_{K^{2}}   \Big(1- e^{-\frac 1{4} (n-3) D\l_{1}( [x_{2},x_{3}]) \l_2(K)^{-2}}\Big )^2
     \l_{1}( [x_{2},x_{3}])^{-2}\ \prod_{i=2}^{3} \dint x_i.
\end{align*}
where $D$ denotes the diameter of $K$.
We apply the affine Blaschke-Petkantschin formula from Lemma \ref{lem:BPF}.
\begin{align*}
\P 
&
	(([X_{1},\ldots,X_{d+1}]\cup [X_{d},\ldots,X_{2d}])^o \cap \xi_{n+1}  = \emptyset) 
    \\ & \leq  
	b_2 c(2,K) (n-3)^{-2} 
	\int\limits_{(\cH_{1}^2)^2} \int\limits_{(H \cap K)^{2}}
	\Big(1- e^{-\frac 1{4} (n-3) D\l_{1}( [x_{2},x_{3}]) \l_2(K)^{-2}}\Big )^2
	\l_{1}( [x_{2},x_{3}])^{-1}\ \prod_{i=2}^{3} \dint x_i  \dint H .
	\end{align*}
The intersection of $H$ and $K$ in a line segment from, say $a_H$ to $b_H$.  Since $K$ is bounded we can parametrize the line segment such that it is contained in $[0, D]$ for all $H$ using the diameter $D$ of $K$. Thus the inner integration is bounded by 
\begin{align*}
\int\limits_{[0,D]^2} &
	\Big(1- e^{-\frac 1{4} (n-3) D\l_{1}( [x_{2},x_{3}]) \l_2(K)^{-2}}\Big )^2
	\l_{1}( [x_{2},x_{3}])^{-1}\ \prod_{i=2}^{3} \dint x_i  
	\\ &\leq 
	2 \int\limits_{0}^{D} \int\limits_{0}^{x_2 }
	\Big(1- e^{-\frac 1{4} (n-3) D t \l_2(K)^{-2}}\Big )^2
	t^{-1}\ \dint t  \dint x_2  
	\\ &\leq 
	2 D \int\limits_{0}^{D}
	\Big(1- e^{-\frac 1{4} (n-3) D t \l_2(K)^{-2}}\Big )^2
	t^{-1} \ \dint t  
	\\ &\leq 
	2 D 
	\int\limits_{0}^{\frac 1n }
	\Big(\frac 1{4} n D t \l_2(K)^{-2} \Big )^2  t^{-1}\ \dint t  
+
	2 D
	\int\limits_{\frac 1n}^{D} t^{-1} \ \dint t  
	\\ &\leq 
	D \Big(\frac 1{4} D \l_2(K)^{-2} \Big )^2  
	+
	2 D (\ln n + \ln D)
	\leq c(2,K) \ln n
	\end{align*}
We put this into 
\begin{align*}
&\E \left( 
\sum_{ i =1}^{n-1}
\1([X_i, X_n, X_{n+1}]^o \cap \xi_{n+1}  = \emptyset)
\right)^{2}  
\\ & = 
(n-1) \P ([X_1,X_2, X_3]^o \cap \xi_{n+1}  = \emptyset)
+
{ n-1 \choose 2}  \P ([X_1,X_2, X_3]^o\cup [X_2, X_3, X_4]^o \cap \xi_{n+1}  = \emptyset)
\\ &  \leq 
C(2,K) \left((n-1)  n^{-1} 
+
{ n-1 \choose 2} (n-3)^{-2} \ln n\right)
\\ & \leq  
C(2,K)  \ln n
\end{align*}
where  we used \eqref{eq:Pemptyn-1} again. Finally, combining this result with \eqref{eq:boundmaxE} and \eqref{eq:Bound} yields
\begin{align}
\E \deg_{1}(\xi_n) \leq c(2,K) \left(n  + n \ln ^{\frac 12} n\right) \leq c(2,K) n \ln^{\frac 12} n
\end{align}
and, hence, concludes  the proof. \qed

\subsection{Appendix}

We define
$$
\cI(K) = 
\int\limits_{(\cH_d)^2} \1(H_1 \cap H_2 \cap K \neq \emptyset) (\sin \varphi)^{-2}  \prod_{i=1}^2 \dint H_i 
$$
where $\varphi$ denotes the angle between $H_1$ and $H_2$.

\begin{lemma}\label{le:valuesin} Assume $K$ is a contained in $R B^d$. Then for $ d \geq 3$ we have
$$
\cI(K)  \leq 
R  \frac {d(d-1)\kappa_d \k_{d-1}}{d-2}  {\bf B} \left(\frac{d}2,\frac{1}2\right) 
	$$
where  $\bf B$ denotes the Beta-function.
\end{lemma}

The hyperplanes $H_i$ are parametrized by their unit normal vector $u_i$ and their distance $t_i$ to the origin. The Haar measure on $\cH_d$ is given by $\dint H_i= \dint t_i \dint u_i$ where $\dint t_i$ is integration with respect to Lebesgue measure on the postive hull $\mathrm{pos}\{u_i\}$ and $\dint u_i$ with respect to Lebesgue measure on $S^{d-1}$. 
By rotational invariance we have 
\begin{align*}
\cI(K) 
& =
\int\limits_{(S^{d-1})^2} \int\limits_{\R_+^2} 
\1((t_1  u_1 + u_1^\perp) \cap (t_2 u_2 + u_2^\perp) \cap R B^d \neq \emptyset) (\sin \varphi(u_1, u_2))^{-2}  \prod_{i=1}^2 \dint t_i  \prod_{i=1}^2 \dint u_i 
\\ & =
d \kappa_d 
\int\limits_{S^{d-1}} \int\limits_{\R_+^2} 
\1((t_1  e_d + e_d^\perp) \cap (t_2 u_2 + u_2^\perp) \cap R B^d \neq \emptyset) (\sin \varphi(e_d, u_2))^{-2}  \prod_{i=1}^2 \dint t_i  \ \dint u_2 .
\end{align*}
Because $(t_1  e_d + e_d^\perp) \cap R B^d $ is contained in a $(d-1)$-dimensional ball of radius $R$ for $t_1 \leq R$, we can estimate this integral by
\begin{align*}
\cI(K) \leq 
d \kappa_d 
\int\limits_{S^{d-1}} \int\limits_{\R_+} 
\1((t_2 u_2 + u_2^\perp) \cap R B^{d-1} \neq \emptyset) (\sin \varphi(e_d, u_2))^{-2}  \prod_{i=1}^2 \dint t_i  \ \dint u_2 .
\end{align*}
The indicator function equals one only if $t_2 \leq R \sin \varphi(e_d, u_2)$. Hence integrating over $\R_+$ gives
\begin{align}\label{eq:sin-int}
\cI(K)  \leq 
d \kappa_d R
\int\limits_{S^{d-1}} (\sin \varphi(e_d, u_2))^{-1}  \ \dint u_2 
=
d \k_d R 
\int\limits_{S^{d-1}} \| u_2|_{\R^{d-1}}\|^{-1 }  \dint u_2
\end{align}
Here we used that 
$ 
\sin \sphericalangle (e_d, u)= \| u|_{\R^{d-1}}\| 
$ 
for $u \in S^{d-1}$, where $u|_{\R^{d-1}}$ denotes the orthogonal projection of $u$ onto the coordinate hyperplane given by $e_d^\perp$. 

Next we use that for a function $f:\R^d \to \R$ which is homogeneous of degree $a>-d$, i.e. $f(tx)=t^af(x)$, application of polar coordinates and then Fubini's Theorem gives
$$
\frac 1{d+a}  \int\limits_{S^{d-1}} f(u) \, \dint u =
\int\limits_{S^{d-1}} \int\limits_0^1 f(u) r^{a+d-1} dr du= 
\int\limits_{B^d} f(x) \, \dint x= 
\int\limits_{-1}^1 \int\limits_{B_z} f(y) \, \dint y \dint z
$$
where $B_z$ is the intersection of $B^d$ with the hyperplane $\{ \langle x, e_d \rangle =z \}$ which is a $(d-1)$-dimensional ball of radius $(1-z^2)^{\frac 12}$. 
We apply this to the function 
$f(x)=\| x|_{\R^{d-1}}\|^{-1 }$,
$$
\int\limits_{S^{d-1}} \| u|_{\R^{d-1}}\|^{-1 } \dint u =
(d-1) \int\limits_{-1}^1 \int\limits_{B_z} \| y|_{\R^{d-1}}\|^{-1}  \, \dint y \dint z,
$$
for $d-1>0$.
Introducing again polar coordinates in $B_z$ yields
$$
\int_{B_z} \| y|_{\R^{d-1}}\|^{-1 } dy =
\int_{S^{d-2}} \int_0^{\sqrt{1-z^2}} r^{d-3 }  dr du =
\frac {(d-1)\k_{d-1}}{d-2} (1-z^2)^{\frac{d-2}2}  
$$
for $d-2 >0$. 
and hence
\begin{align*}
\int\limits_{S^{d-1}} \| u|_{\R^{d-1}}\|^{-1 }  \dint u
& =
\frac {(d-1)\k_{d-1}}{d-2} 
\int\limits_{-1}^1  (1-z^2)^{\frac{d-2}2}   \, \dint z  
\\ & =
 \frac {(d-1)\k_{d-1}}{d-2} 
\int\limits_{0}^1  (1-t)^{\frac{d-2}2} t^{- \frac 12}  \, \dint t  
\\ & =
 \frac {(d-1)\k_{d-1}}{d-2}  {\bf B} \left(\frac{d}2,\frac{1}2\right) . 
\end{align*}
We combine this with \eqref{eq:sin-int} and finally obtain
\begin{align*}
\cI(K)  \leq 
R  \frac {d(d-1)\kappa_d \k_{d-1}}{d-2}  {\bf B} \left(\frac{d}2,\frac{1}2\right) 
\end{align*}
which proves the lemma.\qed

\section*{Acknowledgement}
Part of this work was done during a stay of the first author at the Case Western Reserve University in Cleveland, by invitation by Elisabeth Werner. I thank her for her hospitality, and Mark Meckes for many helpful discussions.

We are indebted to an anonymous referee who pointed out a gap in the prove of subsection \ref{sub34}.

Matthias Reitzner was supported by the DFG via RTG 1916 {\it Combinatorial Structures in Geometry}. Daniel Temesvari was supported by the DFG  via RTG 2131 {\it High-Dimensional Phenomena in Probability -- Fluctuations and Discontinuity}.

\vspace{1cm}

\footnotesize

\textsc{Matthias Reitzner:} Institut f\"ur Mathematik, Universit\"at Osnabr\"uck, Germany \\
\textit{E-mail}: \texttt{matthias.reitzner@uni-osnabrueck.de}

\bigskip

\textsc{Daniel Temesvari:} Institut f\"ur Diskrete Mathematik und Geometrie, Technische Universit\"at Wien, Austria\\
\textit{E-mail}: \texttt{daniel.temesvari@tuwien.ac.at}

\end{document}